\documentclass[12pt]{amsart}
\usepackage{amssymb,epsfig,amsmath,latexsym,amsthm}
\usepackage{graphicx}
\theoremstyle{plain}
\newtheorem{theorem}{Theorem}[section]
\newtheorem{lemma}[theorem]{Lemma}

\newtheorem{proposition}[theorem]{Proposition}

\theoremstyle{definition}

\newtheorem{remark}[theorem]{Remark}

\usepackage{epsfig,color}

\makeatother

\theoremstyle{definition}

\def\fnum{equation}

\numberwithin{equation}{section}

\newcommand{\R}{{\bf{R}}}

\begin{document}

\title{Self-shrinking Platonic solids}

\author{Daniel Ketover}\address{Imperial College London\\Huxley Building, 180 Queen's Gate, London SW7 2RH}

 \email{d.ketover@imperial.ac.uk}

\thanks{I was supported by an NSF Postdoctoral Research Fellowship DMS-1401996 as well as by ERC-2011-StG-278940.}
 
\begin{abstract}  We construct new embedded self-shrinkers in $\R^3$ of genus $3$, $5$, $7$, $11$ and $19$ using variational methods.  Our self-shrinkers resemble doublings of the Platonic solids and were discovered numerically by D. Chopp in 1994.
  \end{abstract}

\maketitle

\setcounter{section}{-1}

\section{Introduction}

A surface $\Sigma\subset\R^3$ is a \emph{self-shrinker} if it satisfies the equation 
\begin{equation}\label{shrinker}
H=\frac{1}{2}\langle x,\nu\rangle,
\end{equation}
where $\nu$ and $H$ denote the normal and mean curvature vector, respectively.  Such surfaces move via homotheties by the mean curvature flow (MCF) and model the singularities that can form along the flow.  

It is well known (Proposition 3.6 in \cite{CM}) that a surface $\Sigma$ satisfying the self-shrinker equation \eqref{shrinker} is equivalent to $\Sigma$ being a minimal surface in the Gaussian metric $(\R^3, e^{-|x|^2/4}\delta_{ij})$.  Equivalently, $\Sigma$ is a critical point for the Gaussian area functional
\begin{equation}\label{Gaussian area}
F(\Sigma)= \frac{1}{4\pi}\int_\Sigma e^{-\frac{|x|^2}{4}}.
\end{equation}

While in standard $\R^3$ one can produce many minimal surfaces using the Weierstrass representation, in the Gaussian metric there is no such method and very few self-shrinkers are known to exist. The simplest shrinkers are the planes through the origin, cylinders and the two-sphere of radius $2$.  Loosely speaking, the plane models smooth points along the flow, the cylinder models neck-pinches and the sphere models the singularities that occur as a convex body shrinks to a round point.  There is also a rotationally symmetric torus discovered by Angenent \cite{A}.

Using PDE gluing methods, Kapouleas-Kleene-M{\o}ller \cite{KKM} and Nguyen \cite{N} constructed examples which have high genus and resemble a desingularization of a sphere and plane. M{\o}ller \cite{M} later desingularized the sphere and Angenent torus to produce compact self-shrinkers of high genus.   

In this paper we construct five new embedded self-shrinkers of low genus using variational methods.  The examples were first discovered numerically in 1994 by Chopp.  In his problem list on MCF, Ilmanen \cite{I} asked whether they could be constructed rigorously. We have the following:

\begin{theorem}\label{main}
There exist an embedded self-shrinker of genus $3$ with tetrahedral symmetry $T_{12}$, embedded self shrinkers of genus $5$ and $7$ with octahedral symmetry $O_{24}$, as well as embedded self-shrinkers of genus $11$ and $19$ with icosahedral symmetry $I_{60}$.
\end{theorem}

Each of the self-shrinkers produced by Theorem \ref{main} can be thought of as a doubling of a Platonic solid.  For the tetrahedral surface,  for instance, consider two smoothed out tetrahedra parallel to each other.  Add in one catenoidal neck at the center of each face joining the two tetrehedra.  Since there are four faces, and the first neck adds no genus, one obtains in this way a surface of genus three, which is isotopic to the self-shrinker constructed in Theorem \ref{main}.  For each of the five surfaces constructed by Theorem \ref{main} the genus is similarly one less than the number of faces of the corresponding Platonic solid.  For illustrations of our self-shrinkers, see Figures 9 and 13 in \cite{C}.

 We produce our new self-shrinkers using variational methods.  They are achieved as min-max limits from a suitable two parameter equivariant min-max procedure.  The main ingredients in the construction are the min-max theory for Gaussian weight  \cite{KZ}, equivariant min-max theory \cite{K2}, the Catenoid Estimate \cite{KMN} and Lusternick-Schnirelman theory.  We also use the recent classification by Brendle of genus zero self-shrinkers \cite{B}.  

Let us give a brief explanation of how, say, the ``doubled cube" of genus $5$ arises variationally.  First consider one-parameter sweepouts of Gaussian $\R^3$ by genus zero surfaces that are $O_{24}$-equivariant (i.e. have the symmetries of the cube).  Running an equivariant min-max procedure, one produces the standard self-shrinking two-sphere.  There is even an optimal foliation $\Phi_t$ of $\R^3$ consisting of round spheres which the min-max limit sits atop.  Now we consider the two parameter family $\Phi_{t,s}$ of genus $5$ surfaces where $\Phi_{t,s}$ is the connect sum of $\Phi_t$ and $\Phi_s$ joined along $6$ necks at the center of each face of the (spherical) cube.  By Lusternick-Schnirelman theory, the width of this two parameter family cannot be equal to the width of the first.  By the Catenoid Estimate, the width corresponding to this two parameter family is strictly less than twice the width of the one parameter family $\Phi_t$ and thus one cannot get as a min-max limit twice the self-shrinking two-sphere.   We then rule out other degenerations of the min-max limit, using in part Brendle's result on uniqueness of genus $0$ self-shrinkers.   It follows that the genus does not degenerate and one produces a min-max minimal surface of genus $5$.

We are unable to tell whether the shrinkers produced by Theorem \ref{main} are compact.  The numerical analysis by Chopp \cite{C} suggests the surfaces are all  compact.  There are on the other hand simple homotopy classes of surfaces consisting of closed surfaces where the min-max limit is non-compact.  For instance, consider the (non-equivariant) saturation coming from the one parameter foliation of $\R^3$ by two-spheres.  The min-max limit is a plane with index $1$, and thus has developed an end (Example 2 in \cite{KZ}).  

To rule out ends one might hope to show that the Gaussian area of spherical caps is strictly less than the area of cones over the same curve, and thus by projecting the min-max sequence into a bounded region one could decrease area.  This however does not seem to be the case, as the area of the cone is comparable to the spherical cap at certain radii.

I am not aware of any gluing construction that has produced surfaces with low genus.  This is one definite advantage of the min-max method in the presence of symmetry. Classification theorems of some sort, however, are typically required.  In this paper we will need to appeal to Brendle's result \cite{B} on uniqueness of genus zero self-shrinkers (Bernstein-Wang \cite{BW} have also proved classification results which give an alternative approach to some of our arguments).  Similarly, in \cite{KMN} in order to double the Clifford torus with very few necks we needed to appeal to the resolution of the Willmore Conjecture \cite{MN}.
 
The organization of this paper is as follows. In Section 1 we introduce the min-max theory for the Gaussian weight in the equivariant setting.  We also introduce the Catenoid Estimate and generalize it slightly to handle two parameters. In Section 2 we prove Theorem \ref{main}. 
\\
\\
{\it Acknowledgements:  I am grateful to Lu Wang and Toby Colding for first making me aware of this problem.  It is a pleasure to thank Fernando Marques, Andr\'{e} Neves, Otis Chodosh and Xin Zhou for several conversations.}

\section{Gaussian Min-max theory}
Together with X. Zhou \cite{KZ}, we extended the classical min-max theory of Simon-Smith \cite{SS} (which refined the earlier approach of Almgren and Pitts \cite{P}) to the case when the ambient metric is the singular Gaussian metric.  

The advantage of the Simon-Smith approach to min-max theory as opposed to that of Almgren-Pitts is that it allows us to consider sweepouts with a fixed topology, and thus we have some hope to control the genus of the limiting minimal surfaces.  Genus bounds for min-max limits were first proved in \cite{SS} in the case of sweepouts of three-spheres by two-spheres.  The argument of Simon-Smith was later generalized by De Lellis-Pellandini \cite{DP}.  Optimal genus bounds as conjectured by Pitts-Rubinstein \cite{PR} were established in \cite{K1}.  

 Let us first give an exposition of min-max theory in the Gaussian setting that we will need in this paper.  First we discuss the notion of a continuous sweepout.   Set $I^n=[0,1]^n$ and let $\{\Sigma_t\}_{t\in I^n}$ be a family of closed subsets of $\R^3$ and $B\subset\partial I^n$.   We call the family $\{\Sigma_t\}$ a \emph{(genus-g) sweepout} if

\begin{enumerate}
\item $F(\Sigma_t)$ is a continuous function of $t\in I^n$,
\item $\Sigma_t$ converges to $\Sigma_{t_0}$ in the Hausdorff topology as $t\rightarrow t_0$.
\item For $t_0\in I^n\setminus B$, $\Sigma_{t_0}$ is a smooth embedded closed surface of genus $g$ and $\Sigma_t$ varies smoothly for $t$ near $t_0$.
\item For $t\in B$, $\Sigma_t$ consists of the union of a $1$-complex together (possibly) with a smooth surface.
\end{enumerate}

Given a sweepout $\{\Sigma_t\}$, denote by $\Pi=\Pi_{\{\Sigma_t\}}$ the smallest saturated collection of sweepouts containing $\{\Sigma_t\}$.  We will call two sweepouts \emph{homotopic} if they are in the same saturated family. 
We define the \emph{width} of $\Pi$ to be
\begin{equation}
W(\Pi,\R^3)=\inf_{\{\Lambda_t\}\in\Pi}\sup_{t\in I^n} F(\Lambda_t).
\end{equation}
A \emph{minimizing sequence} is a sequence of sweepouts $\{\Sigma_t\}^i\in\Pi$ such that 
\begin{equation}
\lim_{i\rightarrow\infty}\sup_{t\in I^n}F(\Sigma_t^i)=W(\Pi,\R^3).
\end{equation}
A \emph{min-max sequence} is then a sequence of slices $\Sigma_{t_i}^i$, $t_i\in I^n$ taken from a minimizing sequence so that $F(\Sigma_{t_i}^i)\rightarrow W(\Pi,\R^3)$.
The main point of the Min-Max Theory \cite{KZ} is that if the width is greater than the maximum of the areas of the boundary surfaces, then some min-max sequence converges to a minimal surface (i.e. a self-shrinker) in $\R^3$.  The statement about the genus collapse was proved in \cite{K1}:

\begin{theorem} (Min-Max Theorem with Gaussian weight \cite{KZ})\label{minmax}
Given a sweepout $\{\Sigma_t\}_{t\in I^n}$ of genus $g$ surfaces, if
\begin{equation}
W(\Pi,\R^3)> \sup_{t\in\partial I^n} F(\Sigma_t),
\end{equation}
then there exists a min-max sequence $\Sigma_i:=\Sigma_{t_i}^i$ such that
\begin{equation}
\Sigma_i\rightarrow k\Gamma \mbox{     as varifolds,} 
\end{equation}
where $\Gamma$ is a connected smooth embedded self-shrinker and $k$ a  positive integer. Moreover, for any compact domain $K\subset\R^3$, after performing finitely many compressions on $\Sigma_i\cap K$ and discarding some components, each connected component of $\Sigma_i\cap K$ is isotopic to $\Gamma\cap K$ implying the genus bound:
\begin{equation}\label{genusbound}
kg(\Gamma)\leq g, 
\end{equation}
where $g(\Gamma)$ denotes the genus of $\Gamma$.
\end{theorem}
\begin{remark}
The fact that the self-shrinker produced by Theorem \label{minmax} is connected follows from the Frankel-type property of Gaussian space: any two self-shrinkers must intersect.  
\end{remark}
We will also need the following equivariant version of Theorem \ref{minmax} from \cite{K2} (as announced in \cite{PR},\cite{PR2}).  Let $G\subset SO(3)$ be a finite group acting on $\R^3$.  A set $\Sigma\subset\R^3$ is called $G$-equivariant if $g(\Sigma)=\Sigma$ for all $g\in G$.  Define the \emph{singular set} $\mathcal{S}$  to be the set of points in $x\in\R^3$ with isotropy group $G_x= \{g\in G\; : gx=x\}$ not equal to the identity $\{e\}$ element. The set $\mathcal{S}$ consists of straight line segments with isotropy $\mathbb{Z}_n$  joining up at the origin at a point with isotropy $\mathbb{D}_n$, $T_{12}$, $I_{24}$, or $I_{60}$.

Suppose $\{\Sigma_t\}_{t\in I^n}$ is an $n$-parameter genus $g$ sweepout of $\R^3$ by $G$-equivariant surfaces so that each surface with positive area intersects $\mathcal{S}$ transversally.  Consider the saturation $\Pi^G_{\Sigma_t}$ of the family $\{\Sigma_t\}$ by isotopies through $G$-equivariant isotopies.  Then we have the following equivariant version of Theorem \ref{minmax}:
\begin{theorem}(\cite{K2})\label{equivariant}
If $W(\Pi^G_{\Sigma_t})>\sup_{t\in\partial I^n} F(\Sigma_t)$, then some min-max sequence converges to a $G$-equivariant smooth self-shrinker in $\R^3$.  The genus bound \eqref{genusbound} also holds and furthermore any compression is $G$-equivariant. 
\end{theorem}

We also will need the Catenoid Estimate that we proved together with F.C. Marques and A. Neves \cite{KMN}.  The Catenoid Estimate asserts that one can sweep out a small tubular neighborhood about an unstable minimal surfaces starting at the boundary of the tubular neighborhood and ending at a $1$-d graph on the minimal surface, with all areas of the sweepout surfaces strictly less than twice that of the minimal surface.  The idea is to use logarithmically cut-off parallel surfaces for the foliation in a $1$-parameter version of the ``log cut-off trick."  In \cite{KMN} we used this estimate to rule out the phenomenon of multiplicity in several different geometric situations.  In this paper we use it in an essential way to rule out that our min-max sequence collapses with multiplicity two to the self-shrinking two-sphere.

We recall the Catenoid Estimate precisely for the reader's convenience.  We will need a slight extension of it from its formulation in \cite{KMN} in order to deal with higher parameter familes.  Let us first introduce the following notation: if $\Sigma$ is a surface in a $3$-manifold $M$, and $\phi>0$ a function defined on $\Sigma$, for $\epsilon,\delta\in\mathbb{R}$ then we can define the $\phi$-adapted tubular neighborhood about $\Sigma$.
\begin{equation}
T_{\delta\phi,\epsilon\phi}:=\{\exp_p(t\phi(p)N(p)):p\in\Sigma, t\in [\delta,\epsilon]\}
\end{equation}
Then we have the following Catenoid Estimate that produces sweepouts of $T_{\delta\phi,\epsilon\phi}$ when $\delta$ and $\epsilon$ have opposite signs:
\begin{proposition}\label{catenoid2} (Theorem 2.3 in \cite{KMN})\label{catenoidinthreemanifold}
Let $M$ be a $3$-manifold and let $\Sigma$ be a closed orientable unstable embedded minimal surface in $M$ of genus $g$.  Denote by $\phi$ the lowest eigenfunction of the stability operator normalized so that $|\phi|_{L^2}=1$. Fix $p_1,p_2...p_k\in\Sigma$ and a graph $\mathcal{G}$ in $\Sigma$ so that there is a retraction $\{R_t\}_{t=0}^1$ of $\Sigma\setminus\{p_1,...p_k\}$ onto $\mathcal{G}$.  Then there exists $\alpha_0>0$ and $\tau_0>0$ so that if $\epsilon>0$ and $\delta<0$ and $\max(|\epsilon|,|\delta|)\leq\alpha_0$, there exists a sweepout $\{\Lambda_t\}_{t=0}^1$ of $T_{\delta\phi,\epsilon\phi}(\Sigma)$ so that:
\begin{enumerate}
\item $\Lambda_t$ is a smooth surface of genus $2g+k-1$ (i.e. two copies of $\Sigma$ joined by $k$ necks)  for $t\in (0,1)$,
\item $\Lambda_0=\partial (T_{\delta\phi,\epsilon\phi}(\Sigma))\cup\bigcup_{i=1}^k\{\exp_{p_i}(s\phi(p_i)N(p_i)): s\in [\delta,\epsilon]\}$
\item $\Lambda_1=\mathcal{G}$
\item For all $t\in [0,1]$, $\mathcal{H}^2(\Lambda_t)\leq 2\mathcal{H}^2(\Sigma)-\tau_0(\epsilon^2+\delta^2)$.
\end{enumerate}  
\end{proposition}
\begin{remark}
Proposition \ref{catenoid2} was proved \cite{KMN} in the special case $\delta=\epsilon$.  The general case of a non-symmetric tubular neighborhood follows with only trivial modifications.  
\end{remark}

Let us exhibit the sweepout that is asserted in Proposition \ref{catenoid2}.  For notational simplicity let us assume $k=1$ and $p_1=:x$.
First define for $0\leq t\leq R$ the logarithmic cutoff function $\eta_t:M\setminus B_{t^2}(x)\rightarrow\mathbb{R}$:
\begin{equation*}
    \eta_t(x) = \begin{cases}
               1      & r(x)\geq t\\
               (1/\log(t))(\log t^2-\log r(x))          & t^2\leq r(x)\leq t.
           \end{cases}
\end{equation*}
It follows from the proof of Theorem 2.3 in \cite{KMN} that $R$ can be chosen sufficiently small independently of $\epsilon$ and $\delta$.  Then we define for $0\leq t\leq R$ the sweepout:
\begin{align}
\Phi_{t}=&\{\exp_p(\epsilon\eta_t(p)\phi(p)N(p)) : p\in\Sigma\setminus B_{t^2}\}\cup\\& \{\exp_p(\delta\eta_t(p)\phi(p)N(p)) : p\in\Sigma\setminus B_{t^2}\},
\end{align}
where $N$ is a choice of unit normal vector on $\Sigma$. 

For $s\in [0,1]$ we then can consider the sweepout
\begin{align}
\Gamma_{s}=&\{\exp_{R_s(p)}(h_1(s)\eta_R(p)\phi(p)N(p)) : p\in\Sigma\setminus B_{R^2}\}\cup\\&\{\exp_{R_s(p)}(h_2(s)\eta_R(p)\phi(p)N(p)) : p\in\Sigma\setminus B_{R^2}\} 
\end{align}
where $h_1$ and $h_2$ are non-negative functions satisfying $h_1(0)=\epsilon$, $h_2(0)=\delta$, $h_1(1)=h_2(1)=0$ and are both decreasing very fast (so that equation 2.61 in \cite{KMN} is satisfied).

Note that the two surfaces $\Phi_R$ and $\Gamma_0$ agree. The desired sweepout $\Lambda_t$ is obtained by concatenating the sweepouts $\Gamma_s$ and $\Phi_t$.  The sweepout $\Phi_t$ ``opens up the hole" a definite amount, and the family $\Gamma_s$ retracts the resulting surface to a graph.

We now state the form of the Catenoid Estimate necessary when $\delta$ and $\epsilon$ have the same sign.  We will have to deal with such situations because of the two parameter families that we consider in this paper.
\begin{proposition}\label{catenoid3} \label{catenoidpos}
Let $M$ be a $3$-manifold and let $\Sigma$ be a closed orientable unstable embedded minimal surface in $M$ of genus $g$.  Denote by $\phi$ the lowest eigenfunction of the stability operator normalized so that $|\phi|_{L^2}=1$. Fix $p_1,p_2...p_k\in\Sigma$ and a graph $\mathcal{G}$ in $\Sigma$ so that there is a retraction $\{R_t\}_{t=0}^1$ of $\Sigma\setminus\{p_1,...p_k\}$ onto $\mathcal{G}$.  Then there exists $\alpha_1>0$ and $\tau_1>0$ so that whenever $\alpha_1\geq\epsilon> \delta\geq 0$ there exists a sweepout $\{\Lambda_t\}_{t=0}^1$ of $T_{\delta\phi,\epsilon\phi}(\Sigma)$ so that:
\begin{enumerate}
\item $\Lambda_t$ is a smooth surface of genus $2g+k-1$ (i.e. two copies of $\Sigma$ joined by $k$ necks)  for $t\in (0,1)$,
\item $\Lambda_0=\partial (T_{\delta\phi,\epsilon\phi}(\Sigma))\cup\bigcup_{i=1}^k\{\exp_{p_i}(s\phi(p_i)N(p_i)): s\in [\delta,\epsilon]\}$
\item $\Lambda_1=\{\exp_{p\in \mathcal{G}}(\delta N(p))\}$
\item For all $t\in [0,1]$, $\mathcal{H}^2(\Lambda_t)\leq 2\mathcal{H}^2(\Sigma)-\tau_1(\delta^2+\epsilon^2)$.
\end{enumerate}  
\end{proposition}
\begin{proof}
Similarly as in Proposition \ref{catenoidinthreemanifold} we can take for $R$ suitably small and $t\in [0,R]$,
\begin{align}
\Phi_{t}=&\{\exp_p((\delta+(\epsilon-\delta)\eta_t(p))\phi(p)N(p)) : p\in\Sigma\setminus B_{t^2}\}\cup\\& \{\exp_p(\delta\phi(p)N(p)) : p\in\Sigma\setminus B_{t^2}\}.\label{second}
\end{align}
Note that the second piece of the surface $\Phi_t$ defined in \eqref{second} is not logarithmically cut off.  For $s\in [0,1]$ we then can consider the sweepout
\begin{align}
\Gamma_{s}=&\{\exp_{R_s(p)}((\delta+h_1(s))\eta_R(p)\phi(p)) : p\in\Sigma\setminus B_{R^2}\}\cup\\&\{\exp_{R_s(p)}(\delta\eta_R(p)\phi(p)) : p\in\Sigma\setminus B_{R^2}\} 
\end{align}
where $h_1$ is a non-negative decreasing function satisfying $h_1(0)=\epsilon-\delta$, $h_1(1)=0$ and so that $h_1$ approaches $0$ at $1$ sufficiently rapidly.

Again since $\Phi_R$ and $\Gamma_0$ agree, by concatenating the sweepouts $\Phi_t$ and $\Gamma_s$ we obtain the desired sweepout $\Lambda_t$.  As in the proof of the Catenoid Estimate (Theorem 2.3 in \cite{KMN}), (4) holds so long as $\alpha_1$ and $R$ are taken sufficiently small.
\end{proof}

Let $f(x)$ be a strictly increasing $C^1$ function defined near $0$ so that $f'\geq C>0$. By the fundamental theorem of calculus, we obtain for any $h>0$, $f(x+h)-f(x)\geq hC$.  Define the surfaces
\begin{equation}
\Lambda(\epsilon,\epsilon+\eta):=\partial (T_{f(\epsilon)\phi,f(\epsilon+\eta)\phi}(\Sigma))\cup\bigcup_{i=1}^k\{\exp_{p_i}(s\phi(p_i)N(p_i)): s\in [f(\epsilon),f(\epsilon+\eta)]\}.
\end{equation}
Let us consider the family of surfaces $\Lambda(\epsilon, \epsilon+\eta)$ as $\epsilon$ varies near $0$.

From the above discussion (combining Proposition \ref{catenoidinthreemanifold} and Proposition \ref{catenoidpos}) we have the following {\it catenoid estimate with parameter} (where $\Sigma$ and $\phi$ are defined as in the statement of Proposition \ref{catenoidinthreemanifold})

\begin{proposition}\label{catenoidoneparam} (Catenoid Estimate with Parameter)
There exists $\alpha_2> 0$ and $\tau_2>0$ so that when $\eta$ is small enough, there is a two-parameter family of surfaces $\Gamma_{s,t}$ parameterized by $(s,t)\in A := [-\alpha_2,\alpha_2-\eta]\times [0,1]$ so that 
\begin{enumerate}
\item $\Gamma_{s,0}=\Lambda(s, s+\eta)$ for $s\in [\alpha_2, \alpha_2-\eta]$
\item The genus of $\Gamma_{s,t}$ is $2g+k-1$ for $(s,t)$ in the interior of $A$
\item $\Gamma_{s,1}$ consists of $1$-d graphs for $s\in [\alpha_2, \alpha_2-\eta]$
\item $\sup_{(s,t)\in A}\mathcal{H}^2(\Gamma_{s,t})\leq 2\mathcal{H}^2(\Sigma)-\tau_2\eta^2$.
\end{enumerate}
\end{proposition}
\begin{proof}
The families produced by Proposition \ref{catenoidpos} and \ref{catenoidinthreemanifold} can be pasted together continuously.  Since $f(s+\eta)-f(s)\geq\eta C$, at least one of the two components of $\Lambda(s, s+\eta)$ has distance from $\Sigma$ at least $C\eta/2$.  This implies (4).
\end{proof}

\begin{remark} We will apply Proposition \ref{catenoidoneparam} in a bounded subset of Gaussian $\R^3$ (where the Gaussian metric is smooth) and where $\Sigma$ is the self-shrinking two-sphere.  Since Proposition \ref{catenoidinthreemanifold} is a purely local statement, there is no need to deal with issues at infinity. 
\end{remark}

\section{Proof of Theorem \ref{main}}
We will construct the genus five surface (i.e. the ``doubled cube") as the other surfaces follow by making trivial changes. 

Consider the octahedral group $O_{24}\subset SO(3)$ acting on $\R^3$ and the $O_{24}$-equivariant one parameter sweepout $\{\Phi_t\}_{t\in [0,1]}$ of $\R^3$ by concentric round two-spheres centered at the origin:   
\begin{equation}
\Phi_t:=\{(x,y,z)\in\mathbb{R}^3\:|\:x^2+y^2+z^2 = (\tan(\pi t/2))^2\}.
\end{equation}
Denote by $\Pi_\Phi$ the saturation of sweepouts obtained from the family $\{\Phi_t\}$ that are also equivariant under the group $O_{24}$.  Each sweepout surface in $\Pi_\Phi$ consists of two-spheres except for at $t=0$ and $t=1$ where the surfaces are trivial.  

We define the equivariant width of this homotopy class to be:
\begin{equation}
\omega_1=\inf_{\Gamma\in\Pi_\Phi}\sup_{t\in [0,1]} F(\Gamma_t)
\end{equation}

By Theorem \ref{equivariant} there exists an $O_{24}$-equivariant self-shrinker $\Sigma_1$ (potentially with multiplicity) whose total mass is $\omega_1$. 

We first observe the following simple lemma:

\begin{lemma}\label{oneparam}
$\Sigma_1=\mathbb{S}_*^2$.
\end{lemma}

\begin{proof}
The genus of the minimal surface $\Sigma_1$ is $0$ by Theorem \ref{equivariant}.  By the classification of Brendle \cite{B}, $\Sigma_1$ is either a plane, cylinder or sphere.  Among these only the sphere is $O_{24}$-equivariant.  Multiplicity is ruled out because the explicit sweepout $\Phi_t$ considered above has strictly less area.  
\end{proof}

\begin{remark}
As an alternative proof to Lemma \ref{oneparam}, observe that Bernstein-Wang \cite{BW} have proved that the shrinker with smallest Gaussian area above the flat planes is the self-shrinking two-sphere $\mathbb{S}^2_*$.   Thus since the foliation by parallel planes is not $O_{24}$-equivariant and $\Phi_t$ is an $O_{24}$-equivariant family, it follows that $\Phi_t$ must be an optimal sweepout. 
\end{remark}

We will now construct a two-parameter family of surfaces $\Phi_{t,s}$ parameterized by $(t,s)\in [0,\infty)\times [0,\infty)$.  Loosely speaking, the family will consist of a sphere at radius $t$, and a sphere at radius $s$ joined by six small necks attached along the coordinate axes $O_{24}$-equivariantly.  Along the diagonal where $t=s$ we ``open up" the $6$ necks and retract the surfaces to a graph in order to obtain a smooth sweepout.  All the surfaces then have genus $5$ except toward parts of the boundary of the sweepout where the surfaces are spheres, and along the diagonal $t=s$ where one has retracted the surfaces to 1-d graphs.  

By the Catenoid Estimate (Proposition \ref{catenoidoneparam}) we can construct such a sweepout with all areas less than $2\omega_1$.   Indeed, the problematic slices are those occurring when $s$ and $t$ are both close to the value $T$ so that $\Phi_T$ is a self-shrinking two-sphere and the sweepout surfaces $\Phi_{t,s}$ are very close to a multiplicity $2$ unstable self-shrinking sphere.  Here we ``open up the holes" and retract to a graph using the Catenoid Estimate to ensure all areas of the sweepout have areas strictly below $2\omega_1$. Let us give more details.  

Consider the round unit sphere $\mathbb{S}^2\subset\R^3$ and the set $S_0$ consisting of the six points on $\mathbb{S}^2$ of intersection with the three coordinate axes.  There exists an $O_{24}$-equivariant family $S_t$ of $6$ closed curves starting at the point curves $S_0$ and ending at twice the one-skeleton $S_1$ of the spherical cube (i.e, the tesselation of $\mathbb{S}^2$ by $6$ squares).  This family gives a retraction of $\mathbb{S}^2\setminus S_0$ onto the one-dimensional graph $S_1$.

For $a,b,t\in [0,1]$ let us denote the truncated cone over $S_t$ by 
\begin{equation}
C_t^{a,b}=\{\lambda x : x\in S_t, \lambda\in\tan([\pi a/2,\pi b/2])\}.
\end{equation} 

For any $s,t \in [0,1]$, denote by $\tilde{\Phi}_{t,s}$ the surface $\Phi_s$ where the six components of $\Phi_s\setminus C_t^{s,s}$ containing the points $C_0^{s,s}$ are removed.  
We can now define a continuously varying three-parameter family of surfaces (for $a,b\in [0,1]$ and $\eta\in [0,1]$)
\begin{equation}
\Phi'_{a,b,\eta}=\tilde{\Phi}_{a,\eta}\cup\tilde{\Phi}_{b,\eta}\cup C^{a,b}_\eta
\end{equation}

The surfaces in the family $\Phi'$ have genus five when $\eta>0$ and $a,b\in (0,1)$ as they consist of the sphere $\Phi_a$ and the sphere $\Phi_b$ where six disks have been removed from each and a conical region connecting them based over $S_t$ glued in.

We can now define a two-parameter family parameterized by $(t,s)\in [0,1]\times [\epsilon_1,1]$ given by:
\begin{equation}\label{twofam}
\Phi_{t,s}=\Phi'_{t(1-s),t+s(1-t),\eta(s,t)},
\end{equation}
where $\epsilon_1$ will be chosen later.  The smooth function $\eta(t,s):[0,1]\times [\epsilon_1,1]\rightarrow [0,1]$ satisfies the following properties:
\begin{enumerate}
\item $\eta(t,\epsilon_1)=0$ and $\eta(t,1)=0$ for all $t$
\item $\eta(0,s)=0$ and $\eta(1,s)=0$ for all $s$
\item $\eta>0$ on $(0,1)\times (\epsilon_1,1)$.
\end{enumerate}
By choosing $\eta(t,s)$ appropriately small where it is non-zero we can also guarantee that for some $C>0$
\begin{equation}\label{area1}
\sup_{(t,s)\in [0,1]\times [\epsilon_1,1]} F(\Phi_{t,s})\leq 2F(\mathbb{S}^2_*)-C\epsilon_1^2.
\end{equation}

For each $t$, the family $s\rightarrow\Phi_{t,s}$ starts at $s=\epsilon_1$ as (nearly) a multiplicity two copy of $\Phi_t$ and then the two spherical pieces start to move off in different directions until finally $s=1$ when one is left again with a graph as the neck region has gone to zero ($\eta(1,s)=0$) and the spherical pieces have disappeared (one to the origin, the other to infinity).  

We will now extend the sweepout $\Phi_{t,s}$ defined on $[0,1]\times [\epsilon_1,1]$ to one (not relabeled) defined on $[0,1]\times [0,1]$ by ``opening up the holes" of the surfaces $[0,1]\times\{\epsilon_1\}$ using the retraction.   The slices that cause a problem are those with areas close to $2F(\mathbb{S}_*^2)$ as ``opening the holes" for such surfaces would push the areas over $2F(\mathbb{S}_*^2)$ if we did not use the Catenoid Estimate.   Denote by $T$ the unique value in $[0,1]$ for which $\tan(\pi T/2)=4$ so that $\Phi_T$ is the self-shrinking $2$-sphere.

We now apply the Catenoid Estimate (Proposition \ref{catenoidoneparam}) with $\Sigma=\mathbb{S}^2_*$, the retraction $S_t$, $\phi=1$, where $f(x)=\tan(\pi/2(T(1-\epsilon_1)+x(1-\epsilon_1)))$ and with $\eta=\epsilon_1/(1-\epsilon_1)$ (shrinking $\epsilon_1$ if necessary).   Let $\alpha_2$ be the constant provided by the Catenoid Estimate.  The sweepout provided by the Catenoid Estimate extends $\Phi_{t, s}$ to the rectangle $A:=[T-\alpha_2,T+\alpha_2-\eta]\times [0,\epsilon_1]$ with the property that for some $C>0$,
\begin{equation}\label{area2}
\sup_{(t,s)\in A} F(\Phi_{t,s})\leq 2F(\mathbb{S}^2_*)-C\epsilon_1^2.
\end{equation}

One needs finally to extend the sweepout $\Phi_{t,s}$ to the regions $[0,T-\alpha_2]\times [0,\epsilon_1]$ and to $[T+\alpha_2-\eta,1]\times [0,\epsilon_1]$.  To handle these regions, observe that for fixed but arbitrary $t\in [0,T-\alpha_2]$ or $t\in[T+\alpha_2-\eta,1]$ both spherical components of $\Phi_{t,\epsilon_1}$ have radius a  definite distance away from that of $\mathbb{S}^2_*$, and thus while the retraction may increase area slightly along the way, if we further decrease $\epsilon_1$, the area of the extension will stay below $2F(\mathbb{S}^2_*)$.  Note that the choice of $\alpha_2$ and equations \eqref{area1}, and \eqref{area2} remain valid if we decrease $\epsilon_1$.  In light of this observation and \eqref{area1}, and \eqref{area2}, we produce a sweepout $\Phi_{t,s}$ with the property that 
\begin{equation}\label{lessthan}
\sup_{(t,s)\in [0,1]\times[0,1]} F(\Phi_{t,s})<2F(\mathbb{S}^2_*).
\end{equation}

Note that the sweepout in the interior of $I^2$ consists of genus $5$ surfaces except at the interval $[0,1]\times\{\epsilon_1\}$ where it consists of two spheres.  By a small perturbation, we can open up the necks so that every surface associated to an interior point of $[0,1]\times [0,1]$ has genus $5$.

\begin{remark}\label{special}
The family $\Phi_{t,s}$ has the following key property.  Modulo the addition of arcs contributing no area, $\Phi_{0,s}=\Phi_s$ and $\Phi_{1,s}=\Phi_{1-s}$.  Thus we have the identification $\Phi_{0,s}=\Phi_{1,1-s}$.  Because of these identifications, we can think of $\Phi_{t,s}$ being a sweepout parameterized by a M{\"o}bius band.  Moreoever, suppose we have a curve $(t(\tau),s(\tau))_{\tau=0}^1$ with the property that $t(0)=0$ and $t(1)=1$ and $s(0)=1-s(1)$. Let us call such a curve $(t(\tau),s(\tau))_{\tau=0}^1$ a  {\it cycle}.  The family of surfaces $\{\Phi_{t(\tau),s(\tau)}\}_{\tau=0}^1$ corresponding to a cycle is a non-trivial sweepout of $\R^3$.  
\end{remark}

Consider the min-max value among $O_{24}$-equivariant two parameter sweepouts homotopic to the family $\Psi_{t,s}$.
\begin{equation}
\omega_2=\inf_{\Gamma\in\Pi_\Psi}\sup_{(t,s)\in [0,1]^2} F(\Phi_{t,s}).
\end{equation}
It follows by definition that $\omega_2\geq\omega_1$.  Moreoever, from \eqref{lessthan} we obtain
\begin{lemma}\label{catenoid}
$\omega_2 < 2\omega_1$.
\end{lemma}

Let us assume for the moment that $\omega_2> \omega_1$.  In this case, since $\omega_1$ is the maximum of areas of slices at the boundary of the sweepout $\Phi$, the Min-Max Theorem \ref{equivariant} applies to produce a self-shrinker $\Sigma_2$ whose total mass (counted with multiplicity $n$) is $\omega_2$.   

From the Min-Max Theorem \ref{minmax} it follows that after finitely compressions on a min-max sequence, and after throwing away finitely many components, the min-max sequence consists of $n$ parallel copies of $\Sigma_2$.  There are only three equivariant compression that can occur.  The first possibility (i) is that the min-max sequence collapses with multiplicity $2$ or $1$ to the sphere.  Multiplicity two is ruled out by Lemma \ref{catenoid} and multiplicity one is ruled out by the assumption $\omega_2>\omega_1$. The second possibility (ii) is that the surfaces collapse into several disjoint spheres, but then by the genus bounds the min-max limit would have to be a multiple of a sphere, which by Brendle \cite{B} must be $\mathbb{S}^2_*$.  This cannot happen since the width is strictly between $\omega_1$ and $\omega_2$.  The third possibility (iii) is that the outer spheres pushes off to infinity and the catenoidal necks push out to become ends.  This results in a genus $0$ surface with $6$ ends.  By the classification of Brendle \cite{B}, this is again impossible since then $\Sigma_2$ would have to be the closed two-sphere $\mathbb{S}^2_*$. The conclusion is that no genus collapse can occur and the genus of $\Sigma_2$ is five.  It also follows that $\Sigma_2$ is achieved with multiplicity $1$.   Of course, it is still possible that $\Sigma_2$ has developed multiple ends.

Let us give more details and show that degenerations i), ii) and iii) are the only possible ones. To do this, it helps to consider the surfaces $\Phi_{t,s}/O_{24}$ projected down to the singular quotient space $\R^3/O_{24}$ via the natural projection $\pi:\R^3\rightarrow\R^3/O_{24}$.  It is then enough to classify compressions in the quotient that lift to compressions in $\R^3$.  These are precisely compressions along curves that bound disks centered about an arc of the singular set, i.e., that intersect the singular set in one point, or about disks that are entirely disjoint from the singular set.  

The space $\R^3/O_{24}$ can be thought of as an orbifold with underlying space $\R^3$ but with a singular set consisting of rays of isotropy $\mathbb{Z}_2$, $\mathbb{Z}_3$ and $\mathbb{Z}_4$ emanating from the origin.  The surfaces $\Phi_{t,s}/O_{24}$ are topologically two-spheres.  To see how they intersect the singular set, consider two concentric spheres centered about the origin in $\R^3/O_{24}$ and so intersecting each singular half-line two times.  Then add a neck along the half-line of $\mathbb{Z}_4$ isotropy to connect the two spheres (so that now the resulting connected sphere is disjoint entirely from the singular set $\mathbb{Z}_4$ but intersects each arc of the singular set with isotropy $\mathbb{Z}_2$ and $\mathbb{Z}_3$ still twice).  Such surfaces are isotopic to $\Phi_{t,s}/O_{24}$.  By adding a point at infinity to we can think of $\Phi_{t,s}/O_{24}$ as genus zero Heegaard splitting of $\mathbb{S}^3\simeq (\R^3/O_{24})\cup\{\infty\}$.  Denote by $B_1$ the three-ball bounded by $\Phi_{t,s}/O_{24}$ that is disjoint from $\{\infty\}$, and by $B_2$ the other three-ball.  

Note that given a closed surface in $\Phi_{t,s}/O_{24}$ of a given genus and intersecting the isotropy subgroups a given number of times, we can compute the genus of the lifted surface using the Riemann-Hurwitz formula.  For instance, consider a two-sphere $\Sigma$ in $\R^3/O_{24}$ intersecting each of the three singular rays with isotropy $\mathbb{Z}_2$,  $\mathbb{Z}_3$, and  $\mathbb{Z}_4$ precisely once.  Then $\pi:\pi^{-1}(\Sigma)\rightarrow \Sigma$ is a branched covering map and one has the following expression for the Euler characteristic of the lifted surface:
\begin{equation}\label{sum}
\chi(\pi^{-1}(\Sigma))=24\chi(\Sigma)-\sum_{\{\pi^{-1}(x):x\in\Sigma\cap\mathcal{S}\}} (|G_x|-1).
\end{equation}

Since $\Sigma$ intersects each singular arc once, there are twelve points in $\pi^{-1}(\Sigma)$ with isotropy $\mathbb{Z}_2$, six with isotropy $\mathbb{Z}_4$ and eight with isotropy $\mathbb{Z}_3$.  Thus we obtain that the term in the summation in \eqref{sum} in the case at hand is $-12*1-6*3-8*2$, so that $\chi(\pi^{-1}(\Sigma))=2$ and $\pi^{-1}(\Sigma)$ is a sphere.  Similarly one can show that a two-sphere in $\R^3/O_{24}$  intersecting any one of the three singular arcs twice and disjoint from the other two lifts to a union of two-spheres. 

Let us now classify the possible compressions of $\Phi_{t,s}/O_{24}$ into $B_1$. Note that $\Phi_{t,s}/O_{24}=\partial B_1=\partial B_2$ intersects the singular set $\mathcal{S}$ in four points: $a$ and $a'$ of type $\mathbb{Z}_2$ and $b$ and $b'$ of type $\mathbb{Z}_3$.  Moreoever $\mathcal{S}\cap B_1$ consists of two unknotted arcs, $A$ joining $a$ to $a'$ and $B$ joining $b$ to $b'$.  A compression can occur along either $A$, $B$ or be completely disjoint from the singular set $A\cup B$.  If the compression occurs along $A$, then afterwards one is left with two spheres $S_1$ and $S_2$ in $B_1$, intersecting in total the singular arc $A$ four times, and intersecting the singular arc $B$ two times.  Since by topological considerations each such sphere must intersect the arcs $A$ and $B$ an even number of times, it follows that $S_1$ intersects $A$ twice and $B$ twice, while $S_2$ only intersects $A$ twice and is disjoint from $B$.  This case corresponds in the lifted picture to the pinching off of a sphere from the genus $5$ surface.  The sphere cannot contribute to the min-max limit (since otherwise the min-max limit would be a $\mathbb{S}^2_*$ with some multiplicity, which is ruled out since $\omega_1<\omega_2<2\omega_1$).  Analogously, a compression along the arc $B$ pinches off spheres in the lifted picture for the same reason.  Finally, one can have a compression that is disjoint from $A$ and $B$.  Since the parity of number of intersection points with $A$ and $B$ must be even, it follows in this case that after compression one has i) two spheres $S_1$, $S_2$, where $S_1$ intersects $A$ twice and is disjoint from $B$ and vice versa for $S_2$ or else ii) $S_1$ intersects $A$ and $B$ twice and $S_2$ is disjiont from $A$ and $B$.  In the lifted picture i) corresponds to the min-max sequence degenerating to a union of two-spheres, which is ruled out as above.  Situation ii) corresponds in the lifted picture to the pinching off of spheres from the genus $5$ surface, which cannot contribute as before.

Let us finally classify compressions into $B_2$.  One can think of the singular set inside $B_2$ as consisting of two points $\{0\}$ and $\{\infty\}$ along with several arcs. There is an arc $C$ with $\mathbb{Z}_4$ isotropy contained in $B_2$ joining $\{0\}$ to $\{\infty\}$.  There are arcs $D$ and $D'$ with $\mathbb{Z}_3$ isotropy emanating from $\{0\}$ and $\{\infty\}$ respectively, both connecting to $\partial B_2$.   Finally there are arcs $E$ and $E'$ with isotropy $\mathbb{Z}_2$ joining $\{0\}$ to $\partial B_2$ and $\{\infty\}$ to $\partial B_2$ respectively.   As before, compressing along $E$, $E'$, $D$ or $D'$ merely splits off a sphere in the lifted picture.  Compressing along $C$ however results in two parallel two spheres, which would give rise to $\mathbb{S}^2_*$ with multiplicity $2$ or $1$. This is ruled out again since $\omega_1<\omega_2<2\omega_1$.

There is one further compression that can occur into $B_2$ which is due to the non-compactness of the ambient space $\R^3$.  Instead of compressing into $C$ to result in two spheres, one can perform the compression in a degenerate way so that the ``outer" sphere enclosing $\{\infty\}$ collapses to $\{\infty\}$ in the process.  This gives rise in the lifted picture to possibility iii), i.e. namely, that the outer sphere pushes out to infinity and the catenoidal necks extend to form ends.  By \cite{B}, this cannot occur.

Note that while we have controlled the genus of the min-max limit, tentacles of the min-max sequence can always press out into containing $\{\infty\}$ which result in adding more and more ends.  We do not have a way to rule this out.

\begin{remark}
For the other four self-shrinkers, the same analysis applies.  Projected to the quotient, the sweepout surfaces consist of two concentric spheres joined by a neck along one of the singular arcs, so that the resulting connected two-sphere intersects each of the two other singular arcs twice.  For instance, the self-shrinker of genus $7$ with symmetry group $O_{24}$ results from adding the neck along the ray of $\mathbb{Z}_3$ isotropy.
\end{remark}

It remains finally to consider the case $\omega_2=\omega_1$.  In this case however Lusternick-Schnirelman theory applies:

\begin{lemma}\label{ls}
If $\omega_2=\omega_1$ then $\R^3$ contains infinitely many minimal $O_{24}$-equivariant genus zero or genus five surfaces of area $\omega_1$.
\end{lemma}

Thus if $\omega_2=\omega_1$, since genus $0$ shrinkers are classified by Brendle \cite{B}, it follows that either $\R^3$ contains infinitely many minimal genus $5$ surfaces or else $\omega_2>\omega_1$ and we are done.   

\begin{remark}
Note that from Corollary 1.2 in \cite{BW} that the only self-shrinker of area $\omega_2=\omega_1$ is $\mathbb{S}^2_*$.  Thus in the equality case, it cannot happen that there are infinitely many genus $5$ shrinkers.  We do not need to use this fact however.
\end{remark}

It remains to prove Lemma \ref{ls}:
\\
\\ \noindent
\emph{Proof of Lemma \ref{ls}:}
Suppose toward a contradiction that $\R^3$ admitted only finitely many minimal genus zero equivariant surfaces.  Denote by $\mathcal{S}$ the set of $O_{24}$-equivariant stationary varifolds whose support is a smooth embedded genus $0$ or $5$ shrinker and whose mass is at most $\omega_2$.  For $\delta>0$, denote by $T_\delta(\mathcal{S})$ the set of 2-varifolds in $\R^3$ that are within $\delta$ of $\mathcal{S}$ in the $\mathbf{F}$-metric (defined in Section 5.1 \cite{KZ}).   By assumption we have a sequence of sweepouts  $\{\Phi^i_{t,s}\}_{(t,s)\in I^2}$ satisfying
\begin{equation}\label{min}
\sup_{(t,s)\in I^2}|\Phi^i_{t,s}|\leq\omega_2+\epsilon_i,
\end{equation}
for $\epsilon_i\rightarrow 0$ as $i\rightarrow\infty$.  

Fix $\delta>0$.  For each $i$, consider the sets in $I^2$ whose corresponding surface is close to a smooth minimal genus $0$ surface:
\begin{equation}
\Omega_i=\{(s,t)\in I^2\; : \mathbf{F}(\Phi^i_{t,s},\mathcal{S})\leq\delta\}
\end{equation}
Note that for each $i$, for some $\eta_i>0$, the sets $[0,1]\times[0,\eta_i]\subset I^2$ and  $[0,1]\times[1-\eta_i,1]\subset I^2$ are both disjoint from $\Omega_i$.

We claim that for $i$ large enough, $\Omega_i$ contains a cycle (in the sense of Remark \ref{special}) $\gamma_i(\tau)\subset I^2$.  If not, then we can find a sequence of closed curves $\lambda_i(\tau)$ in $I^2$ in $I^2\setminus\Omega_i$ starting at $(0,0)$ and ending at $(1,1)$.  The family $\tau\rightarrow \Phi^i_{\lambda_i(\tau)}$ gives a sweepout of $\R^3$.  Since $\omega_1=\omega_2$, we obtain from \eqref{min} that 
\begin{equation}\label{secondmin}
\sup_{\tau\in [0,1]}|\Phi^i_{\lambda_i(\tau)}|\leq\omega_1+\epsilon_i,
\end{equation}

Since the sweepouts $\{\Phi^i_{\lambda_i(\tau)}\}_{\tau=0}^1$ satisfy \eqref{secondmin} they are a minimizing sequence for the one-parameter homotopy class $\Pi_{\Phi}$ that produced $\Sigma_1=\mathbb{S}^2_*$.  Since by definition each surface in $\{\Phi^i_{\lambda_i(\tau)}\}_{\tau=0}^1$ is a definite distance $\delta$ away from the set of smooth minimal surfaces obtainable as limits for almost minimizing sequences, it follows that no min-max sequence obtained from $\{\Phi^i_{\lambda_i(\tau)}\}_{\tau=0}^1$ can be almost minimizing in annuli.  Thus the combinatorial ``pull-tight" argument of Pitts \cite{P} (as formulated in Proposition 2 in \cite{CD}) gives rise to a new minimizing sequence $\{\tilde{\Phi}^i_{\lambda_i(\tau)}\}_{\tau=0}^1$ with all areas strictly below $\omega_1$.  This contradicts the definition of $\omega_1$.   Thus the claim is established, i.e., $\Omega_i\subset I^2$ contains a non-trivial cycle for $i$ large.  

Unraveling this, for each $\delta>0$ we obtain for large $i$ a curve $\gamma_i(\tau)$ in $I^2$ so that the corresponding sweepout $\Phi^i_{\gamma_i(\tau)}$ of $\R^3$  consists of surfaces all within $\delta$ of $\mathcal{S}$ (in the $\mathbf{F}$-metric).   But by Proposition 3.6 in \cite{MN2}, if $\delta$ is chosen sufficiently small, $\Phi^i_{\gamma_i(\tau)}$ is {\it not} a sweepout of $\R^3$ (as it consists of surfaces in a small neighborhood of finitely many surfaces). Thus we have contradicted the assumption that there are only finitely many minimal genus $0$ or $5$ $O_{24}$-equivariant self-shrinkers in $\R^3$ of area $\omega_1$.\qed

\subsection{Further discussion}
Chopp \cite{C} and Ilmanen \cite{I} also showed numerically that one could produce self-shrinkers resembling a doubled plane, where the necks are placed symmetrically around a circle.   One can try the same procedure that we use in this paper.  Namely, first consider one parameter sweepouts by parallel planes which have the appropriate symmetry.  The min-max surface realizing the width for this family is the plane in the appropriate direction.  Then one can take connect sum of parallel affine planes to produce a two parameter family whose width must be strictly less than two by the Catenoid Estimate.  The difficulty is that the min-max limit for this two parameter family could degenerate equivariantly not only to a plane with multiplicity $1$ (where Lusternick-Schnirelman theory applies) but also to a sphere or cylinder.  The cylinder has equivariant index one and no nullity, and since one expects the min-max limit to have index two, one can hope to rule this out.  The sphere however has equivariant index two, and it is not clear to me how to rule it out.

\end{document}